\documentclass{amsart}
\usepackage{mathrsfs}
\usepackage{amssymb}
\usepackage{amsmath}
\def\N{\mathbb{N}}

\theoremstyle{definition}
\newtheorem{definition}{Definition}
\newtheorem{theorem}[definition]{Theorem}
\newtheorem{lemma}[definition]{Lemma}

\newtheorem{corollary}[definition]{Corollary}
\newtheorem{proposition}[definition]{Proposition}

\author[Samuel A.~Alexander]{Samuel A.~Alexander$^1$}
\thanks{$^1$Email: alexander@math.ohio-state.edu}

\begin{document}

\title{Highly Lopsided Information and the Borel Hierarchy}

\begin{abstract}
In a game where both contestants have perfect information, there is a 
strict limit on how perfect that
information can be.  By contrast, when one player is deprived of all information, the limit on the other
player's information disappears, admitting a hierarchy of levels of lopsided perfection of information.
We turn toward the question of when the player with super-perfect information has a winning strategy,
and we exactly answer this question for a specific family of lopsided-information games which we call guessing games.
\end{abstract}

\maketitle

\section{Introduction}

Suppose Alice and Bob are playing an infinite game together and Alice has no information at all about what
moves Bob makes.
Formally this means Alice is only permitted to use strategies which do not depend on Bob's moves.
Under this strong restriction, a strategy for Alice is really just a fixed move-sequence:
loosely speaking, she decides all of her moves before the game ever begins.
This opens the possibility for Bob to have ``better than perfect''
information.  At the highest extreme, we could allow Bob to know Alice's entire move-sequence
before he even makes his first move.  Between that and what is normally called ``perfect'' information,
there is a hierarchy of possible perfection.  If Bob has perfect information in the traditional sense \cite{gale} of
the word, then for every move he makes, he obtains a single $\Delta_0$ fact about Alice's move-sequence.
By contrast,
we could allow that with every move he makes, one $\Delta_1$
fact is revealed to Bob, or one $\Delta_2$ fact, etc.

Suppose $S\subseteq \N^{\N}$ is a fixed subset of Baire space.
The \emph{guessing game} for $S$
is as follows.  
Alice chooses a sequence $f:\N\to\N$, which may or may not be in $S$.
The fact that this is chosen in advance corresponds to Alice having no information about Bob's moves.
Now Bob tries to ``guess'' (formally: play
$1$ or $0$) whether or not $f$ is in $S$.  The terms of Alice's sequence are revealed to him
one-by-one and he gets to revise his guess with each revelation, and he wins if his guesses converge to
the correct answer, otherwise Alice wins.  If Bob has a winning strategy, $S$ is said to be guessable.  In an earlier
paper \cite{alexander}, I demonstrated that $S$ is 
guessable if and only if $S\in\mathbf{\Delta}_2^0$, the boldface
pointclass of the Borel hierarchy.  I will generalize this result to higher-order guessing games
which correspond to one player having more and more lopsided information.

The way in which Alice is forced, by lack of information, to choose her moves
before the game begins, bears some resemblance to an auxiliary game invented by Donald
Martin \cite{martin} in which, at a certain point, one player plays a quasi-strategy and in so doing
locks himself into always playing within that quasi-strategy.

\section{Zeroth-Order Guessability}

In this section I will formally introduce guessability.
The basic definition does not clearly generalize to higher orders, so
an equivalent definition will be proved (using some basic first-order
logic) which generalizes more smoothly.

\begin{definition}
Suppose $S\subseteq \N^{\N}$.  A function $G:\N^{<\N}\to\N$ is said to \emph{guess} $S$
(and we say $G$ is a \emph{guesser} for $S$) if and only if, for every $f\in\N^{\N}$,
\[
\lim_{n\to\infty}G(f(0),\ldots,f(n))=
\left\{\begin{array}{l}
\mbox{$1$, if $f\in S$;}\\ \mbox{$0$, if $f\not\in
S.$}\end{array}\right.\]
If any such $G$ exists, we say $S$ is \emph{guessable}.
\end{definition}

\begin{lemma}
\label{guessinggametheory}
Let $S\subseteq\N^{\N}$.
Suppose Alice and Bob are playing natural numbers, that Bob has perfect information (in the usual sense)
but Alice has no information, and that Bob wins if either
\begin{enumerate}
\item Alice's move-sequence is in $S$ and Bob's moves are eventually always $1$, or
\item Alice's move-sequence is not in $S$, and Bob's moves are eventually always $0$.
\end{enumerate}
Then Bob has a winning strategy if and only if $S$ is guessable.
\end{lemma}

\begin{proof}
If Bob has a winning strategy, then define $G(f(0),\ldots,f(n))$
to be the $n$th move Bob makes according to that strategy, assuming
Alice's move-sequence begins with $(f(0),\ldots,f(n))$.
Conversely, if $S$ is guessable, say with guesser $G$,
a winning strategy for Bob's $n$th move is to play $G(f(0),\ldots,f(n))$
where $f(i)$ denotes Alice's $i$th move.
\end{proof}

In \cite{alexander} I showed that the guessable sets are precisely the $\mathbf{\Delta_2^0}$ sets;
this will also be a special case of a later theorem in the present paper.

I want to give an alternate characterization of guessability which will generalize more easily.
This will require some technical machinery from basic logic.

\begin{definition}
\hfill
\begin{itemize}
\item
By $\mathscr{L}_{\mbox{max}}$ I mean the first-order language
which has a constant symbol $\overline{n}$ for every $n\in\N$; an $n$-ary function
symbol $\tilde{w}$ for every function $w:\N^n\to\N$;
an $n$-ary predicate symbol $\tilde{p}$ for every subset $p\subseteq\N^n$;
a special unary function symbol $\mathbf{f}$; and, for every
function $G:\N^n\times \N^{<\N}\to\N$, an $(n+1)$-ary function symbol $G\circ \mathbf{f}$.
\item
For any $f:\N\to\N$, the structure $\mathscr{M}_f$ for the language $\mathscr{L}_{\mbox{max}}$
is defined as follows.  It has universe $\N$.  It interprets $\overline{n}$, $\tilde{w}$, and
$\tilde{p}$ in the obvious ways.  It interprets $\mathbf{f}$ as $f$.  And for every 
$G:\N^n\times \N^{<\N}\to\N$, $\mathscr{M}_f$ interprets $G\circ\mathbf{f}$ as the function
\[
(G\circ \mathbf{f})^{\mathscr{M}_f}(m_1,\ldots,m_n,m)= G(m_1,\ldots,m_n,f(0),\ldots,f(m)).\]
\item
If $\phi$ is an $\mathscr{L}_{\mbox{max}}$-sentence and $\sigma\in\N^{<\N}$, say that $\phi$
is \emph{determined} by $\sigma$ if for every $f,g\in\N^{\N}$ which
extend $\sigma$, $\mathscr{M}_f\models\phi$ iff $\mathscr{M}_g\models\phi$.
\end{itemize}
\end{definition}

In the following lemma and all of the paper, by ``quantifier-free'' I really mean it:
not even bounded quantifiers are allowed (by contrast, some authors take ``quantifier-free''
to mean ``all quantifiers bounded'').

\begin{lemma}
\label{qfdetermined}
If $\phi$ is a quantifier-free $\mathscr{L}_{\mbox{max}}$-sentence and $f:\N\to\N$, then there
is some $k$ big enough that $(f(0),\ldots,f(k))$ determines $\phi$.
\end{lemma}

\begin{proof}
First, I claim that if $t$ is an $\mathscr{L}_{\mbox{max}}$-term with no free variables
then there is a $k$ large enough that $t^{\mathscr{M}_g}=t^{\mathscr{M}_f}$ whenever $g\in\N^{\N}$
extends $(f(0),\ldots,f(k))$.
This is a straightforward induction on the complexity of $t$.
I omit most cases, but just for one example, suppose $t$ is $(G\circ \mathbf{f})(u_0,\ldots,u_n)$ where $G:\N^n\times \N^{<\N}\to\N$
and where $u_0,\ldots,u_n$ are simpler terms.
By induction, find $k_0,\ldots,k_n$ such that each $u_i^{\mathscr{M}_g}=u_i^{\mathscr{M}_f}$ whenever $g\in\N^{\N}$
extends $(f(0),\ldots,f(k_i))$.
Let $k=\max\{k_0,\ldots,k_n,u_n^{\mathscr{M}_f}\}$.
Suppose $g\in\N^{\N}$ extends $(f(0),\ldots,f(k))$.
Then
\begin{align*}
(G\circ \mathbf{f})(u_0,\ldots,u_n)^{\mathscr{M}_g}
&=
(G\circ \mathbf{f})^{\mathscr{M}_g}(u_0^{\mathscr{M}_g},\ldots,u_n^{\mathscr{M}_g})\\
&=
G(u_0^{\mathscr{M}_g},\ldots,u_{n-1}^{\mathscr{M}_g},g(0),\ldots,g(u_n^{\mathscr{M}_g}))\\
&=
G(u_0^{\mathscr{M}_f},\ldots,u_{n-1}^{\mathscr{M}_f},g(0),\ldots,g(u_n^{\mathscr{M}_f}))\\
&=
G(u_0^{\mathscr{M}_f},\ldots,u_{n-1}^{\mathscr{M}_f},f(0),\ldots,f(u_n^{\mathscr{M}_f}))\\
&=
(G\circ \mathbf{f})(u_0,\ldots,u_n)^{\mathscr{M}_f},\end{align*}
as desired.

From this, the lemma follows by induction on the complexity of $\phi$.
\end{proof}

The above lemma would still hold if we allowed bounded quantifiers, but 
the proof would be more complicated.

\begin{corollary}
\label{opensets}
If $\phi$ is a quantifier-free $\mathscr{L}_{\mbox{max}}$-sentence then \[[\phi]:=\{f\in\N^{\N}\,:\,\mathscr{M}_f\models\phi\}\]
is a clopen subset of Baire space.
\end{corollary}

\begin{proof}
Openness is by Lemma~\ref{qfdetermined}.
Closure follows since
$[\phi]^c=[\neg\phi]$.
\end{proof}

If $f:\N\to\N$ and if $\phi$
is an $\mathscr{L}_{\mbox{max}}$-sentence, I will write $f(\phi)$ for
the number
\[
f(\phi) = \left\{\begin{array}{l}
\mbox{$1$ if $\mathscr{M}_f\models \phi$;}\\
\mbox{$0$ otherwise.}\end{array}\right.\]
In other words, $f(\phi)=1$ if and only if $f\in [\phi]$.

\begin{proposition}
\label{guessabilityalternative}
Suppose $S\subseteq\N^{\N}$.
Then $S$ is guessable if and only if there exists a countable set $\Sigma$ of symbols
of $\mathscr{L}_{\mbox{max}}$, a
listing $\phi_0,\phi_1,\ldots$ of all the quantifier-free sentences
of $\mathscr{L}_{\mbox{max}}\cap\Sigma$, and a function $G:\{0,1\}^{<\N}\to\N$
such that for every $f:\N\to\N$,
\[
\lim_{n\to\infty} G(f(\phi_0),\ldots f(\phi_n))
=
\left\{\begin{array}{l}
\mbox{$1$, if $f\in S$;}\\ \mbox{$0$, if $f\not\in
S.$}\end{array}\right.\]
\end{proposition}

\begin{proof}
($\Rightarrow$) Suppose $S$ is guessable, say with guesser $G_0:\N^{<\N}\to\N$.
Let $\Sigma$ be the symbol-set containing $\mathbf{f}$ and $\overline{n}$ for every $n\in\N$.
Let $\phi_0,\phi_1,\ldots$ be any listing of the quantifier-free $\mathscr{L}_{\mbox{max}}\cap\Sigma$ sentences.
Define $G:\{0,1\}^{<\N}\to\N$ as follows.
Suppose $(p_0,\ldots,p_n)\in\{0,1\}^{<\N}$.
Say a formula $\phi$ \emph{appears} if $\phi=\phi_i$ for some $i\leq n$ and $p_i=1$.
Find a maximum-length sequence $(n_0,\ldots,n_k)$
such that for each $i=0,\ldots,k$, the formula $\mathbf{f}(\overline{i})=\overline{n_i}$ appears,
and let $G(p_0,\ldots,p_n)=G_0(n_0,\dots,n_k)$; if $(n_0,\ldots,n_k)$ is not uniquely
determined or no such nonempty sequence exists, let $G(p_0,\ldots,p_n)=0$.

I claim this witnesses the theorem's conclusion.  Let $f\in S$.
Since $G_0$ guesses $S$, find $n_0$ so big that  $\forall n\geq n_0$,
$G_0(f(0),\ldots,f(n))=1$.  Since $\phi_0,\phi_1,\ldots$ is an exhaustive list, there is some $k_0$ so big
that $(\phi_0,\ldots,\phi_{k_0})$ includes all the sentences $\mathbf{f}(\overline{i})=\overline{f(i)}$ for $i=1,\ldots,n_0$.
For each such $i$, $f(\mathbf{f}(\overline{i})=\overline{f(i)})=1$,
so each formula $\mathbf{f}(\overline{i})=\overline{f(i)}$ \emph{appears} in the definition
of $G(f(\phi_0),\ldots,f(\phi_k))$ for any $k>k_0$.
So for any $k>k_0$, $G(f(\phi_0),\ldots,f(\phi_k))=G_0(f(0),\ldots,f(n))$ for some $n\geq n_0$, so equals $1$.
So $G(f(\phi_0),\ldots,f(\phi_k))\to 1$ as $k\to\infty$, as desired.  A similar argument goes for the case $f\not\in S$.

($\Leftarrow$) Suppose $\Sigma$, $\phi_0,\phi_1,\ldots$, and $G_0:\{0,1\}^{<\N}\to\N$ are as in the theorem's conclusion.
Define $G:\N^{<\N}\to\N$ as follows.
Given any sequence $(m_0,\ldots,m_n)\in\N^{<\N}$, let $k\leq n$ be maximal such that for every $i=0,\ldots,k$,
$\phi_i$ is determined by $(m_0,\ldots,m_n)$ (if there is no such $k$, arbitrarily define
$G(m_0,\ldots,m_n)=0$).
Let $G(m_0,\ldots,m_n)=G_0(f(\phi_0),\ldots,f(\phi_k))$ for any $f:\N\to\N$ extending $(m_0,\ldots,m_n)$,
well-defined since $(m_0,\ldots,m_n)$ 
determines $\phi_0,\ldots,\phi_k$.

I claim $G$ guesses $S$.
Suppose $f\in S$.
By hypothesis, there is some $k_0$ such that $\forall k\geq k_0$, $G_0(f(\phi_0),\ldots,f(\phi_k))=1$.
By Lemma~\ref{qfdetermined}, we can find some $j_0$ such that 
$\phi_0,\ldots,\phi_{k_0}$ are all determined by $(f(0),\ldots,f(j_0))$.
Then for any $j\geq j_0$, $G(f(0),\ldots,f(j))=G_0(f(\phi_0),\ldots,f(\phi_k))$ for some $k\geq k_0$.
So for any such $j$, $G(f(0),\ldots,f(j))=1$, as desired.  The case $f\not\in S$ is similar.
\end{proof}

\begin{corollary}
\label{guessabilityrecast}
Let $S\subseteq\N^{\N}$.
Suppose Alice and Bob are playing natural numbers and neither can see the other's moves.
But at the start of the game, Bob is allowed to choose a countable subset $\Sigma$
of $\mathscr{L}_{\mbox{max}}$ and a listing $\phi_0,\phi_1,\ldots$ of the quantifier-free
$\Sigma$-sentences, and then, on his $n$th move, Bob is told whether or not $\phi_n$
holds of Alice's move-sequence (he is told this by a reliable third party, without
Alice's knowledge).  Suppose the winning conditions are as in Lemma~\ref{guessinggametheory}.
Then Bob has a winning strategy if and only if $S$ is guessable.
\end{corollary}

\begin{proof}
If Bob has a winning strategy, let $\Sigma$ and $\phi_0,\phi_1,\ldots$ be as
dictated by that strategy.  For any $(p_0,\ldots,p_n)\in\{0,1\}^{<\N}$, let
$G(p_0,\ldots,p_n)$ be the move dictated by Bob's strategy assuming
Bob is told that Alice's move-sequence satisfies $\phi_i$ for each $p_i=1$ ($i\leq n$)
and $\neg\phi_i$ for each $p_i=0$ ($i\leq n$).

Conversely, suppose $S$ is guessable.  Let $\Sigma$, $\phi_0,\phi_1,\ldots$, and $G:\{0,1\}^{<\N}\to\N$
be as provided by Proposition~\ref{guessabilityalternative}.
A winning strategy for Bob is to choose $\Sigma$, $\phi_0,\phi_1,\ldots$ at the beginning,
and then always play $G(f(\phi_0),\ldots,f(\phi_n))$ where $f$ is Alice's move-sequence; he can do this
using the information he is given.
\end{proof}

\section{Higher-Order Guessability}

Proposition~\ref{guessabilityalternative} provides a way to generalize guessability, giving us a foothold into a hierarchy of
super-perfect information.
The $\Sigma_n$ and $\Pi_n$ formulas of a language are defined inductively: $\Sigma_0=\Pi_0=\Delta_0$ is the set of quantifier-free formulas
(bounded quantifiers not allowed); having defined $\Sigma_n$ and $\Pi_n$, let
\begin{align*}
\Sigma_{n+1}&=\{\exists x\,\phi\,:\,\mbox{$\phi\in\Pi_n$, $x$ any 
variable}\}\\
\Pi_{n+1}&=\{\forall x\,\phi\,:\,\mbox{$\phi\in\Sigma_n$, $x$ any 
variable}\}.
\end{align*}
An
$\mathscr{L}_{\mbox{max}}$
formula is $\Delta_{n+1}$ if it is equivalent (over all the models $\mathscr{M}_f$) to some $\Sigma_{n+1}$ formula and also to some 
$\Pi_{n+1}$ formula of $\mathscr{L}_{\mbox{max}}$.

\begin{definition}
\label{mthorderguessability}
Let $S\subseteq\N^{\N}$, $m\in\N$.  We say $S$ is \emph{$m$th-order 
guessable} if
there exists a countable set $\Sigma$ of 
$\mathscr{L}_{\mbox{max}}$-symbols,
a listing $\phi_0,\phi_1,\ldots$ of all $\Delta_m$ sentences of $\mathscr{L}_{\mbox{max}}\cap\Sigma$,
and a function $G:\{0,1\}^{<\N}\to\N$
such that for every $f:\N\to\N$,
\[
\lim_{n\to\infty} G(f(\phi_0),\ldots,f(\phi_n))
=\left\{\begin{array}{l}
\mbox{$1$, if $f\in S$;}\\ \mbox{$0$, if $f\not\in
S.$}\end{array}\right.\]
\end{definition}

Thus, Proposition~\ref{guessabilityalternative} can be restated as 
follows: ``$S\subseteq\N^{\N}$ is guessable if and only if it is 
$0$th-order guessable.''

\begin{lemma}
Modify the game in Corollary~\ref{guessabilityrecast}
by changing ``quantifier-free'' to ``$\Delta_m$''.  Then Bob has a winning strategy
if and only if $S$ is $m$th-order guessable.
\end{lemma}

\begin{proof}
Immediate.
\end{proof}

The main theorem of the paper will be that
$m$th-order guessability is equivalent to $\mathbf{\Delta}_{2}^{0}$ if $m=0$
or to $\mathbf{\Delta}_{m+1}^{0}$ if $m\not=0$.  We will begin
working toward that result now.

\begin{definition}
Let $\mathbf{\Delta}_2^{'}=\mathbf{\Delta}_2^0$.  For every $m>2$, define $\mathbf{\Delta}_m^{'}$ as follows:
a set $S\subseteq\N^{\N}$ is in $\mathbf{\Delta}_m^{'}$ if and only if $S$ is a countable union of
countable intersections of $\mathbf{\Delta}_{m-2}^0$ sets and also a countable intersection of countable
unions of $\mathbf{\Delta}_{m-2}^0$ sets.
\end{definition}

\begin{lemma}
\label{deltaprimecalculation}
If $m=2$ then $\mathbf{\Delta}_{m}^{'}=\mathbf{\Delta}_{2}^{0}$.  If $m>2$ then $\mathbf{\Delta}_{m}^{'}=\mathbf{\Delta}_{m-1}^{0}$.
\end{lemma}

\begin{proof}
The $m=2$ case is true by definition.  Suppose $m>2$.
First I claim $\mathbf{\Delta}_{m}^{'}\subseteq\mathbf{\Delta}_{m-1}^{0}$.
Suppose $S$ is $\mathbf{\Delta}_{m}^{'}$.
Then $S=\cap_{i\in\N}\cup_{j\in\N}D_{ij}=\cup_{i\in\N}\cap_{j\in\N}E_{ij}$ where the $D_{ij},E_{ij}$ are $\mathbf{\Delta}_{m-2}^{0}$.
In particular, the $D_{ij}$ are $\mathbf{\Sigma}_{m-2}^{0}$, so for every $i$, $\cup_{j\in\N}D_{ij}$ is $\mathbf{\Sigma}_{m-2}^0$.
Thus $S$ is $\mathbf{\Pi}_{m-1}^0$.  Similarly, since the $E_{ij}$ are $\mathbf{\Pi}_{m-2}^0$, $S$ is $\mathbf{\Sigma}_{m-1}^0$.  So $S$ is
$\mathbf{\Delta}_{m-1}^0$.

Conversely, suppose $S$ is $\mathbf{\Delta}_{m-1}^0$.  Then $S=\cup_{i\in\N}\cap_{j\in\N} S_{ij}=\cap_{i\in\N}\cup_{j\in\N}P_{ij}$ where the $S_{ij}$ 
are
$\mathbf{\Sigma}_{m-3}^0$ (or the $S_{ij}$ are basic-open if $m=3$) and the $P_{ij}$ are $\mathbf{\Pi}_{m-3}^0$
(or complements of basic-open if $m=3$).
So the $P_{ij}$ and $S_{ij}$ are $\mathbf{\Delta}_{m-2}^0$, which shows $S$ is $\mathbf{\Delta}_m^{'}$.
\end{proof}

Say that $S\subseteq\N^{\N}$ is \emph{defined} by an $\mathscr{L}_{\mbox{max}}$ sentence $\phi$
if $S=[\phi]$.
As an example, the set of surjections is defined by $\forall x\,\exists y\,\mathbf{f}(y)=x$.

My interest in defining Borel sets by formulas in a powerful language, as in the following lemma,
is partially influenced by Vanden Boom 
\cite{vandenboom} pp.~276--277.
In \cite{alexanderpreprint} I give a similar result using a weaker but slightly nonstandard language.

\begin{lemma}
\label{formulasforborel}
Let $S\subseteq\N^{\N}$.  For $n>0$, $S$ is $\mathbf{\Sigma}_n^0$ (resp.~$\mathbf{\Pi}_n^0$, $\mathbf{\Delta}_n^0$)
if and only if $S$ is defined by a $\Sigma_n$ (resp.~$\Pi_n$, $\Delta_n$) sentence of $\mathscr{L}_{\mbox{max}}$.
\end{lemma}

\begin{proof}
Write $[f_0]$ for the collection of infinite extensions of a finite sequence $f_0\in\N^{<\N}$.

($\Rightarrow$) Suppose $S$ is $\mathbf{\Sigma}_n^0$.
If $n$ is even, write
$S=\cup_{i_1\in\N}\cdots\cap_{i_n\in\N}[f_{i_1\cdots i_n}]^c$
where each $f_{i_1\cdots i_n}\in\N^{<\N}$ (we can assume the $f_{i_1\cdots i_n}$ are nonempty).
If $n$ is odd, write $S=\cup_{i_1\in\N}\cdots\cup_{i_n\in\N}[f_{i_1\cdots i_n}]$.
Let $\ell:\N^n\to\N$ be defined by letting $\ell(i_1,\ldots,i_n)$ be the length of
$f_{i_1\cdots i_n}$, minus 1.
Define $\tau:N^n\times \N^{<\N}\to\N$ by $\tau(i_1,\ldots,i_n,a_1,\ldots,a_k)=1$ if $(a_1,\ldots,a_k)=f_{i_1\cdots i_n}$,
$\tau=0$ everywhere else.
Then for any $f:\N\to\N$, $f$ extends $f_{i_1\cdots i_n}$
if and only if
\[\tau(i_1,\ldots, i_n, f(0), \ldots, f(\ell(i_1,\ldots, i_n)))=1.\]
So if $n$ is even, then $S$ is defined by the $\mathscr{L}_{\mbox{max}}$ sentence
\[
\exists x_1\cdots \forall x_n (\tau\circ\mathbf{f})(x_1,\ldots,x_n,\tilde{\ell}(x_1,\ldots,x_n))=\overline{0}.\]
And if $n$ is odd, then $S$ is defined by the $\mathscr{L}_{\mbox{max}}$ sentence
\[
\exists x_1\cdots \exists x_n (\tau\circ\mathbf{f})(x_1,\ldots,x_n,\tilde{\ell}(x_1,\ldots,x_n))=\overline{1}.\]
The case $\mathbf{\Pi}_n^0$ case is similar, and the $\mathbf{\Delta}_n^0$ case follows.

($\Leftarrow$) Induction on $n$.
For the base case,
suppose $S$ is defined by (say) the $\Sigma_1$ sentence $\exists x\,\phi$
where $\phi$ is quantifier-free.
Corollary~\ref{opensets} ensures
$[\phi(x|\overline{i})]$
is clopen for any $i\in\N$.
Thus 
$S=\cup_{i\in\N}
[\phi(x|\overline{i})]$
is open, so $\mathbf{\Sigma}_1^0$.
Similarly for the $\Pi_1$ and $\Delta_1$ cases.
With the base case done, the induction case is straightforward.
\end{proof}

\begin{proposition}
\label{piece1}
If $S\subseteq\N^{\N}$ is $\mathbf{\Delta}_{m+2}^{'}$ then $S$ is $m$th-order guessable.
\end{proposition}

\begin{proof}
Case 1: $m>0$.  Then $S$ is a $\cup\cap$ of $\mathbf{\Delta}_m^0$ sets
and also a $\cap\cup$ of $\mathbf{\Delta}_m^0$ sets.
Write $S=\cup_{i\in\N}\cap_{j\in\N}D_{ij}=\cap_{i\in\N}\cup_{j\in\N}E_{ij}$ where the $D_{ij}$, $E_{ij}$ are
$\mathbf{\Delta}_m^0$.
By Lemma~\ref{formulasforborel}, we may find $\Delta_m$ sentences $\sigma_{ij}$ defining each $D_{ij}$, and $\Delta_m$
sentences $\tau_{ij}$ defining each $E_{ij}$.
Let $\Sigma$ be the (countable) set of $\mathscr{L}_{\mbox{max}}$ symbols
appearing in the $\sigma_{ij}$ and $\tau_{ij}$.
Let $\phi_0,\phi_1,\ldots$ be any listing of all the $\Delta_m$ sentences of $\mathscr{L}_{\mbox{max}}\cap\Sigma$.
I shall define a function $G:\{0,1\}^{<\N}\to\N$ which will witness the $m$th-order guessability of $S$.

I'll define $G$ in terms of two functions $\mu,\nu:\{0,1\}^{<\N}\to\N$.
Suppose $(p_0,\ldots,p_n)\in\{0,1\}^{<\N}$.
Say that a sentence $\phi_i$ \emph{appears} if $i\leq n$ and $p_i=1$.
Let $\mu(p_0,\ldots,p_n)$ be the minimum $x\in\N$ such that there is no $y\in\N$ such that
$\neg \sigma_{xy}$ appears.  Let $\nu(p_0,\ldots,p_n)$ be the minimum $x\in\N$ such that there is no
$y\in\N$ such that $\tau_{xy}$ appears.
Finally, let $G(p_0,\ldots,p_n)=1$ if $\mu(p_0,\ldots,p_n)<\nu(p_0,\ldots,p_n)$ and let $G(p_0,\ldots,p_n)=0$
otherwise.

I claim $\Sigma,\phi_0,\phi_1,\ldots,G$ witnesses the $m$th-order guessability of $S$.
First, suppose $f\in S$.  I must show $\lim_{n\to\infty}G(f(\phi_0),\ldots,f(\phi_n))=1$.
Since $f\in S=\cup_{i\in\N}\cap_{j\in\N}D_{ij}$, we have $f\in\cap_{j\in\N}D_{ij}$ for some $i$.
So $f\in D_{ij}$ for every $j$.  Thus $\mathscr{M}_f\models \sigma_{ij}$ for every $j$,
and thus $\neg\sigma_{ij}$ cannot \emph{appear} in the definition of $\mu(f(\phi_0),\ldots,f(\phi_n))$
for any $n$.  Thus $\mu$ is bounded above by $i$.
I claim $\lim_{n\to\infty}\nu(f(\phi_0),\ldots,f(\phi_n))=\infty$, which will show that $\nu$ is
eventually always above $\mu$ and thus that $G$ converges to $1$.
It is enough to let $i\in\N$ be arbitrary and show $\nu(f(\phi_0),\ldots,f(\phi_n))\not=i$ for all
$n$ sufficiently large.
Well, $S^c=\cup_{i\in\N}\cap_{j\in\N}E_{ij}^c$, and $f\not\in S^c$,
so for any arbitrary $i\in\N$, there is some $j$ such that $f\in E_{ij}$, whence $\mathscr{M}_f\models \tau_{ij}$.
Thus, for any $n$ large enough that $\phi_0,\ldots,\phi_n$ includes $\tau_{ij}$,
$\tau_{ij}$ \emph{appears} in the definition of $\nu(f(\phi_0),\ldots,f(\phi_n))$, so $\nu(f(\phi_0),\ldots,f(\phi_n))\not=i$.
There is such a sufficiently large $n$, because $\tau_{ij}$ is $\Delta_m$.

A similar argument shows that $\lim_{n\to\infty} G(f(\phi_0),\ldots,f(\phi_n))=0$ if $f\not\in S$.

Case 2: $m=0$.  This case is similar to Case 1, but instead of writing $S$ as a
$\cup\cap$
of $\mathbf{\Delta}_{m}^0$ sets, write it as a
$\cup\cap$
of complements of basic open sets.  And instead of writing $S$ as a
$\cap\cup$
of $\mathbf{\Delta}_m^0$ sets, write it as a $\cap\cup$ of
basic open sets.  Then take the $\tau_{ij}$ and $\sigma_{ij}$ to be quantifier-free formulas in the obvious way.
\end{proof}

\begin{proposition}
\label{piece2}
If $S\subseteq\N^{\N}$ is $m$th-order guessable, then $S$ is $\mathbf{\Delta}_{m+2}^{'}$.
\end{proposition}

\begin{proof}
Suppose $S$ is $m$th-order guessable.  There is a countable set $\Sigma$ of $\mathscr{L}_{\mbox{max}}$ symbols
and a listing $\phi_0,\phi_1,\ldots$ of all the $\Delta_m$ sentences of $\mathscr{L}_{\mbox{max}}\cap\Sigma$,
and a function $G:\{0,1\}^{<\N}\to\N$ which witnesses the $m$th-order guessability of $S$.
For any $f:\N\to\N$, $f\in S$ if and only if $G(f(\phi_0),\ldots,f(\phi_n))$ is eventually always $1$.
Thus we can write
\begin{align*}
S &=
\bigcup_{i\in\N} \bigcap_{j>i} \, \{f \,:\, G(f(\phi_0),\ldots,f(\phi_j))=1 \}\\
&=
\bigcup_{i\in\N} \bigcap_{j>i}
\,\,\,\,\,\,\,\,\,
\bigcup_{\mbox{$0\leq a_1,\ldots,a_j\leq 1$ and $G(\vec{a})=1$}}
\,\,\,\,\,
\bigcap_{0\leq k\leq j} \, \{f \,:\, f(\phi_k)=a_k\}.
\end{align*}
Now, I claim each set $\{f\,:\,f(\phi_k)=a_k\}$ is $\mathbf{\Delta}_m^0$ if $m>0$, or is clopen if $m=0$.
This is because if $a_k=1$ then $\{f\,:\,f(\phi_k)=a_k\}=\{f\,:\,\mathscr{M}_f\models\phi_k\}$, and if $a_k=0$ then
$\{f\,:\,f(\phi_k)=a_k\}=\{f\,:\,\mathscr{M}_f\models\neg\phi_k\}$.
Either way, we have a $\mathbf{\Delta}_m^0$ set by Lemma~\ref{formulasforborel}
(or a clopen set by Corollary~\ref{opensets}, if $m=0$).
Since $\mathbf{\Delta}_m^0$ is closed under \emph{finite} unions and intersections (as are the clopen sets), I have shown $S$ is a countable 
union
of countable intersections of $\mathbf{\Delta}_m^0$ sets (or of clopen sets if $m=0$).

For the dual situation, note for any $f:\N\to\N$,
saying $G(f(\phi_0),\ldots,f(\phi_n))\rightarrow 1$ is equivalent to saying
$G(f(\phi_0),\ldots,f(\phi_n))=1$ for \emph{infinitely many} values of $n$,
because $\lim_{n\to\infty}G(f(\phi_0),\ldots,f(\phi_n))$ must exist by
Definition~\ref{mthorderguessability}.
Thus
\[
S = \bigcap_{i\in\N} \bigcup_{j>i} \, \{f \,:\, G(f(\phi_0),\ldots,f(\phi_j))=1 \}.\]
Thus $S$ is a countable intersection of countable unions of $\mathbf{\Delta}_m^0$ sets (or of clopen sets if $m=0$).
Put together, $S$ is $\mathbf{\Delta}_{m+2}^{'}$.
\end{proof}

\begin{theorem}
The $0$th-order guessable sets are exactly the $\mathbf{\Delta}_2^{0}$ sets,
and for $m>0$, the $m$th-order guessable sets are exactly the $\mathbf{\Delta}_{m+1}^0$ sets.
\end{theorem}

\begin{proof}
By combining Propositions~\ref{piece1} and \ref{piece2}, for any $m$, the $m$th-order guessable
sets are exactly the $\mathbf{\Delta}_{m+2}^{'}$ sets.
The theorem now follows by Lemma~\ref{deltaprimecalculation}.
\end{proof}

In the proof of Proposition~\ref{piece2} we actually proved slightly more than we needed, which
leads to an unexpected standalone corollary.

\begin{corollary}
If $S=\cup_{i\in\N}\cap_{j\in\N}X_{ij}=\cap_{i\in\N}\cup_{j\in\N}Y_{ij}$,
where the $X_{ij}$ and $Y_{ij}$ are $\mathbf{\Delta}_{n}^0$, $0<n\in\N$,
then there is a single family $Z_{ij}$ of $\mathbf{\Delta}_n^0$
sets such that \[S=\cup_{i\in\N}\cap_{j\in\N}Z_{ij}=\cap_{i\in\N}\cup_{j\in\N}Z_{ij}.\]
\end{corollary}

\begin{proof}
By Proposition~\ref{piece1}, $S$ is $n$th-order guessable.  Proposition~\ref{piece2} gives $Z_{ij}$.
\end{proof}

\section{Acknowledgements}

I want to thank Amit K.~Gupta, Steven VanDendriessche, Timothy J.~Carlson, and Dasmen Teh for
much useful feedback, and especially Mr.~Gupta for catching some mistakes 
in an earlier draft.

\end{document}